\documentclass[12pt, reqno]{amsart}
\usepackage{amsmath, amsthm, amscd, amsfonts, amssymb, graphicx, color}
\usepackage[bookmarksnumbered, colorlinks, plainpages]{hyperref}
\hypersetup{colorlinks=true,linkcolor=red,
anchorcolor=green, citecolor=cyan, urlcolor=red, filecolor=magenta, pdftoolbar=true}

\allowdisplaybreaks 

\newtheorem{theorem}{Theorem}[section]
\newtheorem{lemma}[theorem]{Lemma}
\newtheorem{proposition}[theorem]{Proposition}
\newtheorem{corollary}[theorem]{Corollary}
\theoremstyle{definition}
\newtheorem{definition}[theorem]{Definition}

\newtheorem{remark}[theorem]{Remark}
\numberwithin{equation}{section}
\newcommand{\Mathcal}{\mathcal}
\newcommand{\MATHCAL}{\mathcal}

\begin{document}

\title[von Neumann-Schatten Bessel sequences]{
Multipliers for von Neumann-Schatten Bessel sequences in separable
Banach spaces}

\author[H. Javanshiri and M. Choubin]{Hossein Javanshiri and Mehdi Choubin}

\address{Department of Mathematics, Yazd University,
P.O. Box: 89195-741, Yazd, Iran }
\email{\textcolor[rgb]{0.00,0.00,0.84}{h.javanshiri@yazd.ac.ir}}
\address{Department of Mathematics, Velayat University, Iranshahr, Iran }
\email{\textcolor[rgb]{0.00,0.00,0.84}{m.choubin@gmail.com}}




\begin{abstract}
In this paper we introduce the concept of von Neumann-Schatten Bessel multipliers in separable Banach spaces and obtain some of their properties. Finally, special attention is devoted to the study of invertible Hilbert-Schmidt frame multipliers.
These
results are not only of interest in their own right, but also they pave the way for obtaining
some new results for diagonalization of matrices in finite dimensional setting as well as for dual $HS$-frames. In particular, we show that a $HS$-frame is uniquely determined by the set of its dual $HS$-frames.
\end{abstract} \maketitle


\section{Introduction}

Due to the fundamental works done by
Feichtinger and his coauthors \cite{feichmulti,feichmulti1}, Fourier and Gabor
multipliers were formally introduced and popularized from then on.
Now the theory of Fourier and Gabor multipliers plays an important
role in theoretics and applications; For more information about
the history of this class of operators, some of their properties
and their applications in scientific disciplines and in modern
life the reader can consult Section 1 of the papers \cite{balaz11,bs} and the references (for examples) \cite{bened,cord,cord1,cord2,gro1}.
Balazs \cite{balaz3} extended the notion of Gabor multipliers to
arbitrary Hilbert spaces. In details, he considered
the operators of the form
$${\mathbf{M}}_{m,\Phi,\Psi}(f)=
\sum_{i=1}^\infty m_i\big<f,\psi_i\big>\phi_i
\quad\quad\quad\quad\quad\quad(f\in {\Bbb H}),$$ where
${\Phi}=\{{\phi}_i\}_{i=1}^\infty$ and $\Psi=\{\psi_i\}_{i=1}^\infty$
are ordinary Bessel sequences in Hilbert space $\Bbb H$, and
$m=\{m_i\}_{i=1}^\infty$ is a bounded complex scalar sequence in $\Bbb C$.
It
is worthwhile to mention that this class of operators is not only
of interest for applications in modern life, but also it is of
utmost importance in different branches of linear algebra, matrix analysis and functional analysis.
For example, they are used for the diagonalization of matrices \cite[Definition 3.1]{futa}, the diagonalization of
operators \cite{bhs,cord1,schatten-id} and for solving
approximation problems \cite{cord2,feichmulti0,gro1}. We also
recall that by the spectral theorem, every self-adjoint compact
operator on a Hilbert space can be represented as a multiplier
using an orthonormal system. In addition to all these, multipliers generalize the frame operators, approximately
dual frames \cite{app,japp}, generalized dual frames
\cite[Remark 2.8(ii)]{japp}, atomic systems for subspaces \cite{feichatom,jatom}
and frames for operators \cite{gavro}. Therefore, the study of Bessel
multipliers also leads us to new results concerning dual frames and
local atoms, two concepts at the core of frame theory.

Various generalization of Bessel multipliers have been introduced and
studied in a series of papers recently. This paper
continues these investigations. In details, we investigate Bessel multipliers for von
Neumann-Schatten Bessel sequences in separable Banach spaces.
Here it should be noted that,
just because the study
of this class of operators paves
the way for obtaining some new results for von Neumann-Schatten
frame \cite{von2,von1}, this inspires us to investigate this class of operators
in separable Banach spaces. Let
us recall that the von Neumann-Schatten frames in a separable
Banach space was first proposed by Arefijamaal and Sadeghi \cite{von1}
to deal with all the existing frames as a united object.
In fact, the von Neumann-Schatten frame is an extension of
$g$-frames for Hilbert spaces \cite{sun} and $p$-frames for Banach
spaces \cite{cs}, two important generalization of ordinary frames.


\section{von Neumann-Schatten $p$-Bessel sequences: an overview}

In this section, we give a brief overview of von Neumann-Schatten
$p$-Bessel sequences from
\cite{von1}. Nevertheless, we shall require some facts about the
theory of von Neumann-Schatten $p$-class of operators. For
background on this theory, we use \cite{ringrose} as
reference and adopt that book's notation. Moreover, our notation
and terminology are standard and, concerning frames in Hilbert respectively
Banach spaces, they are in general those of the book \cite{c} respectively the paper
\cite{cs}.

\subsection{von Neumann-Schatten $p$-class of operators}

Let ${\Bbb H}$ be a separable Hilbert space and let $(B({\Bbb H}),\|\cdot\|_{\mathrm{op}})$
denotes the $C^*$-algebra of all bounded linear operators on
${\Bbb H}$. For a compact operator $\mathcal A \in B({\Bbb H})$,
let $s_1(\mathcal A) \geq s_2(\mathcal A) \geq \cdots \geq 0$
denote the singular values of $\mathcal A$, that is, the
eigenvalues of the positive operator $|\mathcal A| = (\mathcal
A^*\mathcal A)^{1\over 2}$, arranged in a decreasing order and
repeated according to multiplicity. For $1 \leq p < \infty$, the
von Neumann-Schatten $p$-class $\mathcal{C}_p({\Bbb H})$ is defined to be
the set of all compact operators $\mathcal A$ for which
$\sum_{i=1}^{\infty} s_i^p(\mathcal A)<\infty$. For $\mathcal
A\in\mathcal{C}_p({\Bbb H})$, the von Neumann-Schatten $p$-norm of $\mathcal
A$ is defined by
\begin{equation}\label{vns-norm}
\Vert \mathcal A \Vert_{\mathcal{C}_p({\Bbb H})} = \Big( \sum_{i=1}^{\infty}
s_{i}^{p}(\mathcal A)\Big)^{1\over p}=\Big(\mathbf{tr}|\mathcal
A|^p\Big) ^{1\over p},
\end{equation}
where $\mathbf{tr}$ is the trace functional which defines as
$\mathbf{tr}(\mathcal A)= \sum_{e\in{\mathcal E}}\langle \mathcal
A(e),e\rangle$ and $\mathcal E$ is any orthonormal basis of $\Bbb H$.
The special case
$\mathcal{C}_1({\Bbb H})$ is called the trace class and $\mathcal{C}_2({\Bbb H})$ is
called the Hilbert-Schmidt class. Recall that an operator $\mathcal A$ is in $\mathcal{C}_p({\Bbb H})$
if and only if $\mathcal A^p\in \mathcal{C}_1({\Bbb H})$. In particular,
$\Vert \mathcal A \Vert_{\mathcal{C}_p({\Bbb H})}^{p}=\Vert \mathcal A^p
\Vert_{\mathcal{C}_1({\Bbb H})}$. It is proved that $\mathcal{C}_p({\Bbb H})$ is a two
sided $*$-ideal of $B({\Bbb H})$ and the finite rank operators are
dense in $(\mathcal{C}_p({\Bbb H}), \|\cdot\|_{\mathcal{C}_p({\Bbb H})})$. Moreover,
for $\mathcal A\in \mathcal{C}_p({\Bbb H})$, one has $\|\mathcal
A\|_{\mathcal{C}_p({\Bbb H})} = \|\mathcal A^*\|_{\mathcal{C}_p({\Bbb H})}$,
$\|\mathcal A\|_{\mathrm{op}}\leq\|\mathcal A\|_{\mathcal{C}_p({\Bbb H})}$ and if
$\mathcal B\in B(\Bbb H)$, then
$$\|\mathcal B\mathcal
A\|_{\mathcal{C}_p({\Bbb H})} \leq \|\mathcal B\|_{\mathrm{op}} \|\mathcal
A\|_{\mathcal{C}_p({\Bbb H})}\quad\quad\quad{\hbox{and}}\quad\quad\quad \|\mathcal A\mathcal B\|_{\mathcal{C}_p({\Bbb H})}
\leq \|\mathcal B\|_{\mathrm{op}}\|\mathcal A\|_{\mathcal{C}_p({\Bbb H})}.$$ In
particular, ${\mathcal{C}_p({\Bbb H})}\subseteq {\mathcal{C}_q({\Bbb H})}$ if $1\leq
p\leq q<\infty$.
We also recall that the
space $\mathcal{C}_2({\Bbb H})$ with the inner product
$[T,S]_{\mathbf{tr}}:=\mathbf{tr}(S^* T)$ is a Hilbert space.

Now, for a fixed $1\leq p<\infty$, we define the Banach space
\begin{align*}
\oplus
\mathcal{C}_p({\Bbb H})=\Big\{{\mathcal A}&=\{\mathcal
A_i\}_{i=1}^\infty:~\mathcal A_i\in \mathcal{C}_p({\Bbb H})\quad\forall
i\in{\Bbb N},~~{\hbox{and}}~\\ &\quad\quad\quad\quad\|{\mathcal A}\|_p:= \Big(
\sum_{i=1}^\infty\|\mathcal A_i\|_{\mathcal{C}_p({\Bbb H})}^{p}\Big)^{1\over p}<\infty
\Big\}.
\end{align*}
In particular, $\oplus\mathcal{C}_2({\Bbb H})$ is a Hilbert space with the inner
product
\[\langle {\mathcal A},{\mathcal A}'\rangle:=\sum_{i=1}^{\infty}[{\mathcal A}_i,{\mathcal A}'_i]_{\mathbf{tr}},\] and so
$\|{\mathcal A}\|_2^2=\langle {\mathcal A},{\mathcal A}\rangle$.

We conclude this subsection by recalling the notion of the tensor product of
two arbitrary elements of ${\Bbb H}$ which will be useful in our subsequent analysis. To this end, suppose that $x, y\in \Bbb H$ and
define the operator $x\otimes y$ on ${\Bbb H}$ by
\[\left( {x \otimes y} \right)\left( z \right) = \left\langle {z,y} \right\rangle x\quad\quad\quad(z\in{\Bbb H}).\]
It is obvious that $\|x\otimes y\| = \|x\| \|y\|$ and the rank of
$x\otimes y$ is one if $x$ and $y$ are non-zero. Moreover,
$$\|x\otimes
y\|_{\mathcal{C}_p({\Bbb H})}=\|x\|
\|y\|\quad\quad\quad {\hbox{and}}\quad\quad\quad\mathbf{tr}(x\otimes y)=\left\langle {{x},{y}}
\right\rangle.$$ Thus $x\otimes y$ is in ${\mathcal{C}_p({\Bbb H})}$ for all
$p\geq 1$. Furthermore, the following
equalities are easily verified:
\begin{align*}
\left( {{x} \otimes {x'}} \right)\left( {{y} \otimes {y'}}
\right) &= \left\langle {{y},{x'}} \right\rangle \left( {{x} \otimes {y'}} \right);\\
{\left( {x \otimes y} \right)^*} &= y \otimes x;\\
T\left( {x \otimes y} \right) &= T(x) \otimes y;\\
\left( {x \otimes y} \right)T &= x \otimes {T^*}(y),
\end{align*}
where $x',  y'
\in {\Bbb H}$ and $T\in B({\Bbb H})$. Clearly,
the operator $x \otimes x$ is a rank-one projection if and only
if $\left\langle {{x},{x}} \right\rangle = 1$,
that is, $x$ is a unit vector. Conversely, every rank-one projection is of the
form $x \otimes x$ for some unit vector $x$.

Having reached this state it remains to recall the definition and some properties
of von Neumann-Schatten $p$-frames for separable Banach spaces. This is the subject matter of the next subsection.

\subsection{von Neumann-Schatten $p$-frames}

To simplify the later discussion,
we make the following blanket assumption.\\

$\mathbf{Convention.}$ {\it For the rest of this paper we assume that
$\mathcal H$ is a Hilbert space with orthonormal basis ${\mathcal E}=\{e_i\}_{i\in I}$,
$1<p<\infty$ and $q$ is the conjugate
exponent to p, that is, $1/p+1/q=1$. Moreover, the notation ${\mathcal C}_p$ {\rm[}respectively, ${\mathcal C}_q${\rm]} is used to denote the
space ${\mathcal C}_p({\mathcal H})$ {\rm[}respectively, ${\mathcal C}_q({\mathcal H})${\rm]} without explicit reference to the Hilbert space $\mathcal H$.}\\

Recall from \cite{von1} that a countable family
$\mathcal G=\{\mathcal G_i\}_{i=1}^\infty$ of bounded linear
operators from separable Banach space $\mathcal X$ to $\mathcal{C}_p$ is a von Neumann-Schatten $p$-frame for
$\mathcal X$ with respect to $\mathcal H$ if
constants $0<A_{\mathcal G}\leq B_{\mathcal G}<\infty$ exist such that
\begin{equation}
A_{\mathcal G}\Vert f \Vert_{\mathcal X} \leq \left( \sum_{i=1}^\infty\Vert
{\mathcal G}_i(f)\Vert_{\mathcal{C}_p}^{p}\right)^{1\over p}\leq
B_{\mathcal G}\Vert f \Vert_{\mathcal X}
\end{equation}
for all $f \in \mathcal X$. It is called a von Neumann-Schatten
$p$-Bessel sequence with bound $B_{\mathcal G}$ if the second inequality holds.
In \cite{von1}, the authors have shown that the von
Neumann-Schatten $p$-frame condition is satisfied if and only if
$\left\{{\mathcal A}_i\right\}_{i=1}^\infty\mapsto
\sum_{i=1}^{\infty} {\mathcal A}_i \mathcal G_i$ is a well defined
mapping from $\oplus{\mathcal{C}_q}$ onto $\mathcal X^*$. Motivated by this fact, they considered the following operators:
\begin{gather}
U_{\mathcal G}:\mathcal X\rightarrow \oplus {\mathcal{C}_p};\quad
f\mapsto \left\{\mathcal
G_i(f)\right\}_{i=1}^\infty,\label{analysis}
\end{gather}
and
\begin{gather}
~~~~\quad T_{\mathcal G}:\oplus{\mathcal{C}_q}\rightarrow \mathcal X^*;\quad\left\{{\mathcal
A}_i\right\}_{i=1}^\infty\mapsto \sum_{i=1}^{\infty} {\mathcal
A}_i \mathcal G_i.\label{synthesis}
\end{gather}
As usual, the operator $U_{\mathcal G}$ is called the analysis
operator, and $T_{\mathcal G}$ is the synthesis operator of
$\mathcal G$.

Recall also from \cite{von1} that ${\mathcal G}$ is called a von
Neumann-Schatten $q$-Riesz basis for $\mathcal X^*$ with respect
to $\mathcal H$ if
\begin{enumerate}
\item $\{f\in \mathcal X : \mathcal G_i(f)= 0 \quad\forall i\in
\mathbb{N}\} = \{0\}$, \item there are positive constants $A_{\mathcal G}$ and
$B_{\mathcal G}$ such that for any finite subset $I\subseteq \mathbb{N}$ and
$\{\mathcal A_i\}_{i=1}^\infty\in \oplus \mathcal{C}_q$
\begin{equation*}
A_{\mathcal G}\left(\sum_{i\in I}\|\mathcal A_i\|_{\mathcal{C}_q}^q\right)^{1\over q}
\leq \left\|T_{\mathcal G}( \{A_i\}_{i=1}^\infty)\right\|_{{\mathcal X}^*}
\leq B_{\mathcal G}\left(\sum_{i\in I}\|\mathcal A_i\|_{\mathcal{C}_q}^q\right)^{1\over q}.
\end{equation*}
\end{enumerate}
The assumption of latter definition implies that
$\sum_{i=1}^\infty\mathcal A_i\mathcal G_i$ converges
unconditionally for all ${\mathcal A}=\{\mathcal A_i\}_{i=1}^\infty\in \oplus
\mathcal{C}_q$ and
\begin{equation}\label{11695}
A_{\mathcal G}\|\mathcal A\|_q \leq
\|T_{\mathcal G}(\mathcal A)\|_{{\mathcal X}^*}
\leq B_{\mathcal G}\|\mathcal A\|_q.
\end{equation}
Thus ${\mathcal G}\subseteq B(\mathcal
X,\mathcal{C}_p)$ is a von Neumann-Schatten $q$-Riesz basis for
$\mathcal X^*$ with respect to $\mathcal H$ if and only if the
operator $T_{\mathcal G}$ defined in \eqref{synthesis} is both
bounded and bounded below. Specially, in this case the operators $T_{\mathcal G}$ and
$U_{\mathcal G}$ are bijective.

The reader will remark that if
$\mathcal H={\Bbb C}$, then $B(\mathcal H)={\mathcal C}_p={\Bbb
C}$ and thus $\oplus \mathcal C_p=\ell^p$. Hence the above definitions is consistent with the
corresponding definitions in the concept of $p$-frames for
separable Banach spaces.

We conclude this section with the following result which can be proved with a similar argument as in the proof of \cite[Corollary 2.5]{cs}.

\begin{lemma}\label{riesz-frame}
Let $\mathcal X$ be a reflexive Banach space and let ${\mathcal
G}=\{{\mathcal G}_i\}_{i=1}^{\infty}\subseteq B(\mathcal
X,\mathcal{C}_p)$ be a von Neumann-Schatten $q$-Riesz basis for
$\mathcal X^*$ with respect to $\mathcal H$. If the $q$-Riesz basis
bounds of $\mathcal G$ are $A_{{\mathcal G}}$ and $B_{{\mathcal G}}$, then ${\mathcal G}$ is a von Neumann-Schatten
$p$-frame for $\mathcal X$ with $p$-frame bounds $A_{{\mathcal G}}$ and $B_{{\mathcal G}}$.
\end{lemma}


\section{von Neumann-Schatten Bessel multipliers: Basic results}

All over in this section $\mathcal X_1$ and $\mathcal X_2$ are
separable Banach spaces and the space $\ell^r$ ($1\leq r\leq\infty$) has its
usual meanings.

Our starting point of this section is the following lemma which play
a crucial rule in this paper.

\begin{lemma}\label{upper-b}
Let ${\mathcal G}\subseteq
B(\mathcal X_1^*,\mathcal{C}_p)$ be a von Neumann-Schatten
$p$-Bessel sequence with bound $B_{\mathcal G}$ and ${\mathcal
F}\subseteq B(\mathcal
X_2,\mathcal{C}_q)$ be a von Neumann-Schatten $q$-Bessel sequence
with bound $B_{\mathcal F}$. If ${\frak m}=\{{\frak m}_i\}_{i=1}^{\infty}\in
\ell^\infty$, then the operator
$\mathbf{M}_{{\frak m},\mathcal F,\mathcal G}: \mathcal X_1^*
\rightarrow \mathcal X_2^*$ defined by
\[\mathbf{M}_{{\frak m},\mathcal F,\mathcal G}(f)=\sum_{i=1}^{\infty} {\frak m}_i
\mathcal G_i(f)\mathcal F_i\quad\quad\quad(f\in \mathcal X_1^*),\]
is well-defined and $\|\mathbf{M}_{{\frak m},\mathcal F,\mathcal
G}\|_{\mathrm{op}}\leq B_{\mathcal F} B_{\mathcal G} \|{\frak m}\|_\infty$.
\end{lemma}
\begin{proof}
It is easy to check that $\{{\frak m}_i {\mathcal
G}_i(f)\}_{i=1}^\infty\in \oplus\mathcal{C}_p$ for all $f\in \mathcal X_1^*$.
On the other hand
\[\mathbf{M}_{{\frak m},\mathcal F,\mathcal G}(f)=T_{\mathcal F}\Big(\{{\frak m}_i
\mathcal G_i(f)\}_{i=1}^\infty\Big)\quad\quad\quad (f\in \mathcal X_1^*).\] It
follows that $\mathbf{M}_{{\frak m},\mathcal F,\mathcal G}$ is
well defined. Moreover, we observe that
\begin{align*}
\Big\|\sum_{i=1}^{n} {\frak m}_i \mathcal G_i(f)\mathcal
F_i\Big\|_{\mathrm{op}}&=\sup_{h\in \mathcal X_2, \|h\|\leq 1}
\Big\{{\Big|}\sum_{i=1}^n{\mathbf{tr}}{\Big(}{\frak m}_i{\mathcal G}_i(f){\mathcal F}_i(h){\Big)}{\Big|}\Big\}\\
&\leq\sup_{h\in X_2, \|h\|\leq 1}
\Big\{\sum_{i=1}^n{\Big|}{\mathbf{tr}}{\Big(}{\frak m}_i{\mathcal G}_i(f){\mathcal F}_i(h){\Big)}{\Big|}\Big\}\\
&\leq\sup_{h\in X_2, \|h\|\leq 1}
\Big\{\sum_{i=1}^n{\Big\|}{\frak m}_i{\mathcal G}_i(f){\mathcal F}_i(h){\Big\|}_{\mathcal{C}_1}\Big\}\\
&\leq\sup_{h\in X_2, \|h\|\leq 1}
\Big\{\sum_{i=1}^n\|{\frak m}_i{\mathcal G}_i(f)\|_{\mathcal{C}_p}\|{\mathcal F}_i(h)\|_{\mathcal{C}_q}\Big\}\\
&\leq \|{\frak m}\|_\infty \Big( \sum_{i=1}^n\Vert {\mathcal G}_i(f)\Vert_{\mathcal{C}_p}^{p}\Big)^{1\over p} \sup_{h\in X_2, \|h\|\leq 1}\Big( \sum_{i=1}^n\Vert {\mathcal F}_i(h)\Vert_{\mathcal{C}_q}^{q}\Big)^{1\over q}\\
&\leq B_{\mathcal F} B_{\mathcal G} \|{\frak m}\|_\infty \|f\|,
\end{align*}
for all $n\in{\Bbb N}$ and $f\in X_1^*$. From this, we can deduced
that the operator $\mathbf{M}_{{\frak m},{\mathcal F},{\mathcal
G}}$ is bounded with $B_{\mathcal F} B_{\mathcal G} \|{\frak m}\|_\infty$.
\end{proof}


Now we are in position to introduce the main object of study of
this work.

\begin{definition}\label{p-q-multi}
Let ${\mathcal G}\subseteq B(\mathcal
X_1^*,\mathcal{C}_p)$ be a von Neumann-Schatten $p$-Bessel
sequence, and let ${\mathcal F}\subseteq
B(\mathcal X_2,\mathcal{C}_q)$ be a von Neumann-Schatten
$q$-Bessel sequence. Let ${\frak m}\in
\ell^\infty$. The operator
$\mathbf{M}_{{\frak m},{\mathcal F},{\mathcal G}}: \mathcal X_1^*
\rightarrow \mathcal X_2^*$ defined by
\[\mathbf{M}_{{\frak m},{\mathcal F},{\mathcal G}}(f)=\sum_{i=1}^{\infty} {\frak m}_i {\mathcal G}_i(f){\mathcal F}_i\quad\quad\quad(f\in\mathcal X_1^*)\]
is called von Neumann-Schatten ($p,q$)-Bessel multiplier and the
sequence ${\frak m}$ is called its symbol.
\end{definition}


If $\frak m$ is a sequence in $\ell^r$
($1\leq r\leq\infty$), then the mapping
$${\mathcal M}_{p,{\frak
m}} : \oplus\mathcal{C}_p \rightarrow \oplus\mathcal{C}_p;~\{\mathcal A_i\}_{i=1}^\infty\mapsto
\{{\frak m}_i \mathcal
A_i\}_{i=1}^\infty,$$ is well-defined and bounded. Hence, the von
Neumann-Schatten ($p,q$)-Bessel multiplier $\mathbf{M}_{{\frak
m},{\mathcal F},{\mathcal G}}$ can be written as
$$\mathbf{M}_{{\frak m},{\Mathcal F},{\Mathcal G}} = T_{\Mathcal
F} {\mathcal M}_{p,{\frak m}} U_{\Mathcal G}.$$

This equality paves the way for the study of some operator properties of
$\mathbf{M}_{{\frak m},{\Mathcal F},{\Mathcal G}}$ in terms of the properties of
${\mathcal M}_{p,{\frak m}}$. To this end, we need the following remark and lemma.

\begin{remark}\label{basis}
Recall from \cite{ringrose} that
$\{e_n\otimes e_m:~n, m\in I\}$ is an orthonormal basis of
$\mathcal C_2$. For the convenience
of citation and a better exposition we denote by $\{\mathcal
E_k\}$ the orthonormal basis of ${\mathcal C}_2$. Hence, putting
$$F_{i,k}=\{\delta_{i,j}\mathcal{E}_k\}_{j=1}^\infty\quad\quad\quad
{\Big(}i, k\in{\Bbb N}{\Big)},$$
one can see that
\begin{enumerate}\item in the case where $\dim {\mathcal H}=N$, then ${\Big\{
\{F_{i,k}\}_{k=1}^{{N^2}}\Big\}}_{i=1}^\infty$ is an
orthonormal basis for $\oplus \mathcal{C}_2$.
\item in the case where ${\mathcal H}$ is an infinite
dimensional Hilbert space, then ${\Big\{
\{F_{i,k}\}_{k=1}^{\infty}\Big\}}_{i=1}^\infty$ is an
orthonormal basis for $\oplus \mathcal{C}_2$.
\end{enumerate}
\end{remark}

In what follows, the notation $\overline{{\frak m}}$ is used to denote the sequence $\{\overline{{\frak m}}_i\}_{i=1}^\infty$, where ${\overline{{\frak m}}_i}$
refers to the complex conjugate of ${\frak m}_i$.

\begin{lemma}\label{sym-m}
The following assertions hold.
\begin{enumerate}
\item If ${\frak m}\in\ell^\infty$, then
$\|{\MATHCAL M}_{p,{\frak m}}\|_{\mathrm{op}}=\|{\frak
m}\|_\infty$. \item ${\MATHCAL M}_{2,{\frak m}}^*={\MATHCAL
M}_{2,\overline{{\frak m}}}$. \item
If $\dim \mathcal H=N$ and ${\frak m}\in\ell^p$, then ${\MATHCAL
M}_{2,{\frak m}}\in{\mathcal C}_p(\oplus
\mathcal{C}_2)$ and $$\|{\MATHCAL M}_{2,{\frak
m}}\|_{\mathcal{C}_p(\oplus
\mathcal{C}_2)}=N^2\|{\frak m}\|_p.$$
\end{enumerate}
\end{lemma}
\begin{proof}
(1) That $\|{\MATHCAL M}_{p,{\frak m}}\|_{\mathrm{op}}\leq
\|{\frak m}\|_\infty$ is trivial. In order to prove that $\|{\frak
m}\|_\infty\leq\|{\MATHCAL M}_{p,{\frak m}}\|_{\mathrm{op}}$,
suppose that $j\in {\Bbb N}$ and $x\in \mathcal H$ with $\|x\|=1$.
Observe that, if ${\mathcal A}^{(j)}=\{\delta_{i,j}\cdot x\otimes x\}_{i=1}^\infty$, then $\|{\mathcal A}^{(j)}\|_{p}=1$ and thus
$$
\|{\MATHCAL M}_{p,{\frak m}}\|_{\mathrm{op}}\geq \|{\frak m}
{\mathcal A}^{(j)}\|_p\geq\|{\frak m}_j\cdot x\otimes
x\|_{\mathcal{C}_p}=|{\frak m}_j|.
$$
It follows that $\|{\frak
m}\|_{\infty}\leq
\|{\MATHCAL M}_{p,{\frak m}}\|_{\mathrm{op}}$.

(2) It is suffices to note that if ${\mathcal A},{\mathcal B}\in
\oplus\mathcal{C}_2$, then
\begin{align*}
\langle {\MATHCAL M}_{2,{\frak m}} {\mathcal A},{\mathcal
B}\rangle&=\sum_{i=1}^\infty [{\frak m}_i{\mathcal A}_i,{\mathcal
B}_i]_{\mathbf{tr}}\\&=\sum_{i=1}^\infty [{\mathcal A}_i,\overline{{\frak m}}_i{\mathcal
B}_i]_{\mathbf{tr}}\\&=\langle {\mathcal A},{\MATHCAL M}_{2,
\overline{{\frak m}}}{\mathcal B}\rangle.
\end{align*}

(3) If $\mathcal A\in \oplus
\mathcal{C}_2$, then we have
\begin{align*}
{\MATHCAL M}_{2,{\frak m}}(\mathcal A)&={\MATHCAL M}_{2,{\frak
m}}\left(\sum_{i=1}^\infty\sum_{k=1}^{N^2}\big<\mathcal
A,F_{i,k}\big>F_{i,k}\right)\\
&=\sum_{k=1}^{N^2}\sum_{i=1}^\infty{\frak m}_i\big<\mathcal
A,F_{i,k}\big>F_{i,k}.
\end{align*}
Hence, if we set
$$\hat{F}_j=F_{l+1,j-l{N^2}}\quad\quad\quad\quad
\Big(l\in{\Bbb {N}}\cup\{0\},~ l{N^2}+1\leq
j\leq(l+1){N^2}\Big),$$ and
$$\hat{{\frak m}}=\{\hat{{\frak m}}_j\}_{j=1}^\infty=\{\underbrace{{\frak
m}_1,\cdots,{\frak m}_1}_{N^2},\underbrace{{\frak
m}_2,\cdots,{\frak m}_2}_{N^2},\cdots\},$$ then we observe that
$${\MATHCAL M}_{2,{\frak m}}(\mathcal A)=\sum_{j=1}^\infty\hat{{\frak m}}_j\big<\mathcal
A,\hat{F}_j\big>\hat{F}_{j}.$$ Therefore, ${\MATHCAL M}_{2,{\frak
m}}$ is in the Schatten $p$-class of $\oplus\mathcal{C}_2$;
this is because of, $$\sum_{j=1}^\infty|\hat{{\frak m}}_j|^p\leq
N^2\|{{\frak m}}\|_p^p.$$

Now, in order to prove that $\|{\MATHCAL M}_{2,{\frak
m}}\|_{\mathcal{C}_p}={N^2}\|{\frak m}\|_p$, first note that
$|{\MATHCAL M}_{2,{\frak m}}|={\MATHCAL M}_{2,{|{\frak m}|}}$. Hence, we observe that
\begin{align*}
\|{\MATHCAL M}_{2,{\frak
m}}\|_{\mathcal{C}_p}^p=\mathbf{tr}(|{\MATHCAL M}_{2,{\frak
m}}|^p) &=\sum_{i=1}^{\infty}\sum_{k=1}^{{N^2}}\left\langle
{\MATHCAL M}_{2,{|{\frak m}|^p}}
(F_{i,k}),F_{i,k}\right\rangle\\
&=\sum_{j=1}^{\infty}\left\langle {\MATHCAL M}_{2,{|\hat{{\frak
m}}|^p}}
(\hat{F}_j),\hat{F}_j\right\rangle\\
&=\sum_{j=1}^{\infty} \sum_{k=1}^{{N^2}} \left[|{\frak m}_j|^p
\mathcal{E}_k,\mathcal{E}_k\right]_{\mathbf{tr}}\\
&={N^2}\|{\frak m}\|_p^p.
\end{align*}
We have now completed the proof of the lemma.
\end{proof}


By applying Lemma \ref{sym-m} with $\mathcal H={\Bbb C}$, one
can obtain the following improvement of \cite[Lemma
5.4(3)]{balaz3}.

\begin{lemma}
If ${\frak m}\in\ell^p$, then the
operator
$$\mathcal M_{{\frak
m}}:\ell^2\rightarrow\ell^2;~\{c_i\}_{i=1}^\infty \mapsto\{{\frak
m}_ic_i\}_{i=1}^\infty,$$ is in the Schatten
$p$-class of $\ell^2$. In particular, $\|\mathcal M_{{\frak
m}}\|_{\mathcal{C}_p(\ell^2)}=\|{\frak m}\|_p$.
\end{lemma}


As was mentioned in section 2, in the case where $p=2$, the spaces
${\mathcal C}_2$ and $\oplus_2{\mathcal C}_2$ are Hilbert. Motivated by this fact the authors of \cite{von2,von1} provided a detailed
study of the duals of a von Neumann-Schatten 2-frame for Hilbert space
$\mathcal K$ with respect to $\mathcal H$. Let us recall from
\cite{von1} that, a sequence $\mathcal G\subseteq B(\mathcal
K,{\mathcal C}_2)$ is said to be a Hilbert--Schmidt frame or simply a $HS$-frame for $\mathcal K$
with respect to $\mathcal H$, if there exist two positive
numbers $A_{\mathcal G}$ and $B_{\mathcal G}$ such that
\begin{equation}
A_{\mathcal G}\Vert f \Vert^2_{\mathcal K} \leq
\sum_{i=1}^\infty\Vert {\mathcal
G}_i(f)\Vert_{\mathcal{C}_2}^{2}\leq B_{\mathcal G}\Vert f
\Vert^2_{\mathcal K}\nonumber
\end{equation}
Particularly, by using the Hilbert properties of the spaces, they
observed that
$$
U_{\mathcal G}(f)=\left\{\mathcal
G_i(f)\right\}_{i=1}^\infty\quad\quad\quad{\hbox{and}}\quad\quad\quad
T_{\mathcal G}(\{{\mathcal A}_i\}_{i=1}^\infty)=
\sum_{i=1}^{\infty}\mathcal G_i^*{\mathcal A}_i,$$ where
$f\in{\mathcal K}$ and $\{{\mathcal
A}_i\}_{i=1}^\infty\in\oplus{\mathcal C}_2$. Moreover, they showed that the mapping
$S_{\mathcal G}:=T_{\mathcal G}U_{\mathcal G}$ is a bounded,
invertible, self-adjoint and positive operator, and they called the $HS$-frame
$\widetilde{\mathcal G}:=\{{\mathcal G}_iS^{-1}_{\mathcal G}\}_{i=1}^\infty$
the canonical dual $HS$-frame of $\mathcal G$. It is
worthwhile to mention that the $HS$-frames is a more general version
of $g$-frames, an important generalization of ordinary frames.

The following remark is a very
useful tool in our study of $HS$-Bessel multiplier
when $\mathcal H$ is a finite dimensional space.

\begin{remark}\label{multi}
Suppose that ${\mathcal G}$ and $\mathcal F$ are $HS$-Bessel
sequences for $\mathcal K$ with respect to $\mathcal H$ and that
${\frak m}\in\ell^\infty$. For each
$f\in{\mathcal K}$, we observe that
\begin{align*}
\mathbf{M}_{{\frak m},\mathcal F,\mathcal G}(f)&=T_{\Mathcal
F} {\mathcal M}_{p,{\frak m}} U_{\Mathcal G}(f)\\&=\sum_{i=1}^\infty
{\frak m}_i{\mathcal F}_i^*{\mathcal G}_i(f)\\
&=\sum_{i=1}^\infty\sum_{n,m\in I}{\frak m}_i\big<f,{\mathcal
G}^*_i(e_n\otimes e_m)\big>{\mathcal F}_i^*(e_n\otimes e_m).
\end{align*}
In particular, if $\dim({\mathcal H})=N$, then
\begin{align*}
\mathbf{M}_{{\frak m},\mathcal F,\mathcal G}=\sum_{n,m=1}^N\sum_{i=1}^\infty{\frak m}_i\Big({\mathcal F}_i^*(e_n\otimes e_m)\otimes{\mathcal
G}^*_i(e_n\otimes e_m)\Big).
\end{align*}
Hence, if we set
$\Phi={\Big\{
\{{\mathcal F}_i^*(e_n\otimes e_m)\}_{n,m=1}^{{N^2}}\Big\}}_{i\in \mathbb{N}}$, $\Psi={\Big\{
\{{\mathcal G}_i^*(e_n\otimes e_m)\}_{n,m=1}^{{N^2}}\Big\}}_{i\in \mathbb{N}}$
and
$$\hat{{\frak m}}=\{\underbrace{{\frak
m}_1,\cdots,{\frak m}_1}_{N^2},\underbrace{{\frak
m}_2,\cdots,{\frak m}_2}_{N^2},\cdots\},$$
then $\Phi$ and $\Psi$ are ordinary Bessel sequences and the operator $\mathbf{M}_{{\frak m},\mathcal F,\mathcal G}$ is equal to the operator
$M_{\hat{{\frak m}},\Phi,\Psi}$ in the sense of Balazs \cite[Definition 5.1]{balaz3}
for ordinary Bessel sequences. The reader will remark that
in this case $\mathcal G$ and $\mathcal F$ are $HS$-Riesz basis if and only if
$\Phi$ and $\Psi$ are ordinary Riesz basis, see \cite[Theorem 3.3]{von2}.
\end{remark}


In the case where $\dim({\mathcal H})<\infty$, Remark \ref{multi} paves the way for obtaining some properties of $HS$-Bessel multipliers
from \cite{balaz3,bs}. In details, as an application of this remark and Lemma \ref{sym-m}, by a
method similar to that of \cite[Theorems 6.1 and 8.1]{balaz3} one can easily
obtain the following generalization of those theorems. The details
are omitted.

Let ${\mathcal F}^{(l)}=
\{{\mathcal F}_i^{(l)}\}_{i=1}^\infty$  be a sequence in
$B({\mathcal K},{\mathcal C}_2)$
indexed by $l\in{\Bbb N}$. We say that:
\begin{enumerate}
\item The sequence ${\mathcal F}^{(l)}$ converges uniformly to some sequence
$\mathcal F\subseteq B({\mathcal K},{\mathcal C}_2)$
with respect to operator norm, if for $i\longrightarrow\infty$ we have
$$\sup_l\{\|{\mathcal F}_i^{(l)}-{\mathcal F}_i\|_{\mathrm{op}}\}\longrightarrow 0.$$
\item The sequence ${\mathcal F}^{(l)}$ converges to some sequence
$\mathcal F\subseteq B({\mathcal K},{\mathcal C}_2)$ in $\ell^2$-sense if
$$\forall\varepsilon>0\quad\exists N\in{\Bbb N}\quad{\hbox{such~that~}}\quad
\forall l\geq N\quad\quad{\Big(}\sum_{i=1}^\infty\|{\mathcal F}_i^{(l)}-{\mathcal F}_i\|_{{\mathcal C}_p}^2{\Big)}^{1/2}<\varepsilon.$$
\end{enumerate}

\begin{proposition}\label{gelar2895}
Suppose that ${\mathcal G}$ and $\mathcal F$ are $HS$-Bessel
sequences for $\mathcal K$ with respect to $\mathcal H$ and that $\dim \mathcal H<\infty$. Then the following assertions hold.
\begin{enumerate}
\item If ${\frak m}\in c_0$, then $\mathbf{M}_{{\frak m},\mathcal
F,\mathcal G}$ is compact. \item If ${\frak m}\in\ell^p$, then
$\mathbf{M}_{{\frak m},\mathcal F,\mathcal G}$ is in the Schatten
$p$-class of $\mathcal K$ and $$\|\mathbf{M}_{{\frak m},\mathcal
F,\mathcal G}\|_{{\mathcal C}_p({\mathcal K})}\leq\sqrt{B_{\mathcal G}
B_{\mathcal F}}(\dim \mathcal H)^2\|{\frak m}\|_p.$$
\item Let ${\frak m}^{(l)}=\{{\frak m}_i^{(l)}\}_{i=1}^\infty$ be a sequence
indexed by $l$
converge to $\frak m$ in $\ell^{p}$, then $$\|\mathbf{M}_{{\frak m}^{(l)},{\Mathcal
F},{\Mathcal G}}-\mathbf{M}_{{\frak m},{\Mathcal
F},{\Mathcal G}}\|_{{\mathcal C}_p({\mathcal K})}\longrightarrow 0.$$
\item For the sequences ${\mathcal F}^{(l)}\subseteq B({\mathcal K},{\mathcal C}_2)$
indexed by $l\in{\Bbb N}$, we can say that
\newcounter{j59}
\begin{list}%
{\rm(\alph{j59})}{\usecounter{j59}}
\item If ${\frak m}\in\ell^1$ and the sequence ${\mathcal F}^{(l)}$
is a $HS$-Bessel sequence converging uniformly to $\mathcal F$ with respect to operator norm, then
$$\|\mathbf{M}_{{\frak m},{\Mathcal
F}^{(l)},{\Mathcal G}}-\mathbf{M}_{{\frak m},{\Mathcal
F},{\Mathcal G}}\|_{{\mathcal C}_1({\mathcal K})}\longrightarrow 0.$$
\item If ${\frak m}\in\ell^2$ and the sequence ${\mathcal F}^{(l)}$
converge to $\mathcal F$ in an $\ell^2$-sense, then
$$\|\mathbf{M}_{{\frak m},{\Mathcal
F}^{(l)},{\Mathcal G}}-\mathbf{M}_{{\frak m},{\Mathcal
F},{\Mathcal G}}\|_{{\mathcal C}_2({\mathcal K})}\longrightarrow 0.$$
\end{list}
\item For $HS$-Bessel sequences ${\mathcal G}^{(l)}$ converge to $\mathcal G$,
corresponding properties as in ${\mathrm(4)}$ apply.
\item \newcounter{jparsa59}
\begin{list}%
{\rm(\alph{jparsa59})}{\usecounter{jparsa59}}
\item Let ${\frak m}^{(l)}\longrightarrow{\frak m}$ in $\ell^1$, ${\mathcal F}^{(l)}$
and ${\mathcal G}^{(l)}$ be $HS$-Bessel sequences with bounds $B_{{\mathcal F}^{(l)}}$ and $B_{{\mathcal G}^{(l)}}$ such that $\sup_l B_{{\mathcal F}^{(l)}}<\infty$ and $\sup_l B_{{\mathcal G}^{(l)}}<\infty$. If the sequences ${\mathcal F}^{(l)}$
and ${\mathcal G}^{(l)}$ converge uniformly to $\mathcal F$ respectively $\mathcal G$ with respect to operator norm, then
$$\|\mathbf{M}_{{\frak m}^{(l)},{\Mathcal
F}^{(l)},{\Mathcal G}^{(l)}}-\mathbf{M}_{{\frak m},{\Mathcal
F},{\Mathcal G}}\|_{{\mathcal C}_1({\mathcal K})}\longrightarrow 0.$$
\item Let ${\frak m}^{(l)}\longrightarrow{\frak m}$ in $\ell^2$ and let
${\mathcal F}^{(l)}$
respectively ${\mathcal G}^{(l)}$ converge to $\mathcal F$ respectively $\mathcal G$ in an $\ell^2$-sense, then
$$\|\mathbf{M}_{{\frak m}^{(l)},{\Mathcal
F}^{(l)},{\Mathcal G}^{(l)}}-\mathbf{M}_{{\frak m},{\Mathcal
F},{\Mathcal G}}\|_{{\mathcal C}_2({\mathcal K})}\longrightarrow 0.$$
\end{list}
\end{enumerate}
\end{proposition}


By another application
of Remark \ref{multi} with the aid of \cite[Theorem 5.1]{bs} we have the following generalization of that theorem. In this proposition and in the sequel the sequence $\frak m$ is called
semi-normalized if $0<\inf_i|{\frak m}_i|\leq\sup_i|{\frak m}_i|<\infty$;
In this case, the notation $1/{\frak m}$ is used to denote the sequence $\{1/{\frak m}_i\}_{i=1}^\infty$.

\begin{proposition}\label{parsa2895}
Suppose that ${\mathcal F}$ is a $HS$--Riesz basis for $\mathcal K$ with respect to $\mathcal H$ and that $\dim \mathcal H<\infty$. Then the following
assertions hold.
\begin{enumerate}
\item If $\mathcal G$ is a $HS$-Riesz basis, then
$\mathbf{M}_{{\frak m},\mathcal F,\mathcal G}$ is invertible on
$\mathcal K$ if and only if $\frak m$ is semi--normalized. \item
If $\frak m$ is semi--normalized, then $\mathbf{M}_{{\frak
m},\mathcal F,\mathcal G}$ is invertible on $\mathcal K$ if and
only if $\mathcal G$ is a $HS$--Riesz basis.
\end{enumerate}
\end{proposition}


There does not seem to
be an easy way to extend Propositions \ref{gelar2895} and \ref{parsa2895} to infinite dimensional case.
However, we have the
following result.

\begin{proposition}\label{parsa28952}
Let $\mathcal X_1$ be a reflexive Banach space and ${\mathcal G}\subseteq B(\mathcal
X_1^*,\mathcal{C}_p)$ be a von Neumann-Schatten $q$-Riesz basis
for $\mathcal X_1$ with bounds $A_{\mathcal G}$ and $B_{\mathcal G}$. Let also ${\Mathcal F}\subseteq B(\mathcal
X_2,\mathcal{C}_q)$ be a von Neumann-Schatten $p$-Riesz basis for
$\mathcal X_2^{*}$ with bounds $A_{\mathcal F}$ and $B_{\mathcal F}$. Then for each
${\frak m}\in \ell^\infty$ we
have
\[A_{\mathcal F} A_{\mathcal G}\|{\frak m}\|_\infty\leq
\|\mathbf{M}_{{\frak m},{\Mathcal F},{\Mathcal G}}
\|_{\mathrm{op}}\leq B_{\mathcal F} B_{\mathcal G} \|{\frak m}\|_\infty.\] In particular, the operator $\mathbf{M}_{{\frak m},{\Mathcal F},{\Mathcal G}}$ is invertible and the
mapping ${\frak m}\mapsto\mathbf{M}_{{\frak m},{\Mathcal
F},{\Mathcal G}}$ from $\ell^\infty$ into
$B(\mathcal X_1^*,\mathcal X_2^*)$ is injective.
\end{proposition}
\begin{proof} By Lemma \ref{upper-b}, it will be enough to prove that
we have the lower bound. To this end, suppose that $x$ is an
arbitrary element of $\mathcal H$ with $\|x\|=1$ and $j\in{\Bbb N}$. Then,
${\mathcal A}^{(j)}=\{\delta_{i,j}\cdot x\otimes x\}_{i=1}^\infty\in\oplus\mathcal{C}_p$.
Thus, by the surjectivity of the operator $U_{\mathcal G}$, there
exists an element $f_j\in\mathcal X^*_1$
such that $U_{\mathcal G}(f_j)={\mathcal A}^{(j)}$.
In particular,
$\sum_{i=1}^\infty
 \Vert {\mathcal G}_i(f_j)\Vert_{\mathcal{C}_p}^{p}=1$.
We now invoke Lemma \ref{riesz-frame} to conclude that
$\frac{1}{B_{\mathcal G}}\leq\|f_j\|_{\mathcal X_1^*}\leq \frac{1}{A_{\mathcal G}}$.
Also, since ${\mathcal F}\subseteq B(\mathcal
X_2,\mathcal{C}_q)$ is a Von Neumann-Schatten $p$-Riesz basis for
$\mathcal X_2^*$, Eq. (\ref{11695}) implies that
\begin{equation*}
\|T_{\mathcal F}{\mathcal M}_{p,{\frak m}}({\mathcal A}^{(j)})\| \geq A_{\mathcal F}
\left\|{\mathcal M}_{p,{\frak m}}({\mathcal A}^{(j)})\right\|_p = A_{\mathcal F}|{\frak m}_j|.
\end{equation*}
Now, we get
$$
\|\mathbf{M}_{{\frak m},{\Mathcal F},{\Mathcal G}}\|_{\mathrm{op}}
\geq \sup_{j} {\|T_{\Mathcal
F} {\mathcal M}_{p,{\frak m}} U_{\Mathcal G}(f_j)\| \over \|f_j\|}\geq A_{\mathcal F} A_{\mathcal G}\|{\frak m}\|_\infty,$$
and this completes the proof.
\end{proof}

\section{Invertibility of $HS$--frame multipliers}

In this section, we turn our attention to the study of invertible $HS$-frame multipliers.
First let us to note that, if $\frak m$ is semi-normalized and
$\mathbf{M}_{{\frak m},\mathcal
F,\mathcal G}$ is invertible, then the $HS$-Bessel sequences $\mathcal F$ and $\mathcal G$
are automatically $HS$-frames. We also recall that a $HS$-frame ${\mathcal G}^d=\{{\mathcal G}_i^d\}_{i=1}^\infty$ is called a dual of $\mathcal G$ if $T_{\mathcal G}U_{{\mathcal G}^d}=Id_{\mathcal K}$.

The main result of this section is the following theorem.

\begin{theorem}\label{main}
Suppose that ${\mathcal G}=\{{\mathcal G}_j\}_{j=1}^\infty$ and $\mathcal F=\{{\mathcal F}_j\}_{j=1}^\infty$ are $HS$-frames
for $\mathcal K$ with respect to $\mathcal H$, and
that the symbol $\frak m$ is semi--normalized. If $\mathbf{M}_{{\frak m},\mathcal
F,\mathcal G}$ is an
invertible multiplier, then there exists a unique bounded operator
$\Gamma:{\mathcal K}\rightarrow \oplus_2{\mathcal C}_2$ such that
\begin{equation}\label{9895}
\mathbf{M}^{-1}_{{\frak m},\mathcal
F,\mathcal G}=\mathbf{M}_{1/{\frak m},\widetilde{\mathcal G},{{\mathcal F}^d}}+
\Gamma^*U_{{\mathcal F}^d},
\end{equation}
for all dual $HS$-frames ${\mathcal F}^d=\{{\mathcal F}_j^d\}_{j=1}^\infty$ of $\mathcal F$.
\end{theorem}
\begin{proof} Define $\Gamma:{\mathcal K}\rightarrow \oplus_2{\mathcal C}_2$ by
\begin{equation}\label{111}
\Gamma(f):=U_{\mathcal F}(\mathbf{M}^{-1}_{{\frak m},\mathcal
F,\mathcal G})^*(f)-{\mathcal M}_{2,1/\overline{\frak m}}U_{\mathcal G} S^{-1}_{\mathcal G}(f)\quad\quad\quad(f\in
{\mathcal K}).
\end{equation}
Then the operator $\Gamma$ is
bounded and
$$\mathbf{M}^{-1}_{{\frak m},\mathcal
F,\mathcal G}T_{\mathcal F}=\Gamma^*+S^{-1}_{\mathcal G} T_{\mathcal G}{\mathcal
M}_{2,1/{\frak m}}.$$
Using any dual $HS$-frame ${\mathcal F}^d$ of ${\mathcal F}$ we get
\begin{equation}\label{222}
\mathbf{M}^{-1}_{{\frak m},\mathcal
F,\mathcal G}=S^{-1}_{\mathcal G} T_{\mathcal G}{\mathcal
M}_{2,1/{\frak m}}U_{{\mathcal F}^d}+ \Gamma^*U_{{\mathcal F}^d}.
\end{equation}
It follows that
$\mathbf{M}^{-1}_{{\frak m},\mathcal
F,\mathcal G}={\mathbf{M}}_{1/{\frak m},\widetilde{\mathcal G},{{\mathcal F}^d}}+
\Gamma^*U_{{\mathcal F}^d}$ for all dual $HS$-frames ${\mathcal F}^d$ of ${\mathcal F}$.
Having reached this state it remains to show that
$\Gamma$ is uniquely determined. To this end, suppose on the
contrary that Eq. (\ref{222}) are hold for two operators
$\Gamma_1$ and $\Gamma_2$. It follows that
\begin{equation}\label{gelar}
(\Gamma_1-\Gamma_2)^*U_{{\mathcal F}^d}=0,
\end{equation}
for all dual $HS$-frames ${\mathcal F}^d$ of ${\mathcal F}$. In particular, for each $f\in\mathcal K$, we have
\begin{equation}\label{parsa}
(\Gamma_1-\Gamma_2)^*(\{{\mathcal F}_jS_{\mathcal F}^{-1}(f)\}_{j=1}^\infty)=0.
\end{equation}
Now, let $i$ and $k$ be arbitrary elements in ${\Bbb N}$ and $F_{i,k}=\{\delta_{i,j}\mathcal{E}_k\}_{j=1}^\infty$, which introduced in Remark \ref{basis}. If for each $j\in{\Bbb N}$, we define
$${\mathcal F}'_{i,j,k}:\mathcal K\rightarrow{\mathcal C}_2;\quad f\mapsto\big<f,e'_1\big>\delta_{i,j}\mathcal{E}_k,$$
where $\{e'_l\}_{l\in L}$ is an orthonormal basis for $\mathcal K$.
Then it is not hard to check that the sequence ${\mathcal F}'_{i,k}=\{{\mathcal F}'_{i,j,k}\}_{j=1}^\infty$ is a $HS$-Bessel sequence for $\mathcal K$ with respect to $\mathcal H$ and the $HS$-Bessel sequence
$${\mathcal F}_{i,k}^d=\{{\mathcal F}_jS^{-1}_{\mathcal F}+{\mathcal F}'_{i,j,k}-{\mathcal F}_j
S^{-1}_{\mathcal F}T_{\mathcal F}U_{{\mathcal F}_{i,k}'}\}_{j=1}^\infty$$
is a dual $HS$-frame of ${\mathcal F}$. Therefore, Eq. (\ref{gelar}) and (\ref{parsa}),
implies that
\begin{align*}
0&=(\Gamma_1-\Gamma_2)^*{U_{{\mathcal F}_{i,k}^d}}(f)\\
&=(\Gamma_1-\Gamma_2)^*\left(\{{\mathcal F}_jS^{-1}_{\mathcal F}(f)+{\mathcal F}'_{i,j,k}(f)-{\mathcal F}_j
S^{-1}_{\mathcal F}T_{\mathcal F}U_{{\mathcal F}_{i,k}'}(f)\}_{j=1}^\infty\right)\\
&=(\Gamma_1-\Gamma_2)^*\left(\{{\mathcal F}'_{i,j,k}(f)\}_{j=1}^\infty\right)
\end{align*}
for all $f\in\mathcal K$. Hence, we have
$$0=(\Gamma_1-\Gamma_2)^*\left(\{{\mathcal F}'_{i,j,k}(e'_1)\}_{j=1}^\infty\right)=
(\Gamma_1-\Gamma_2)^*(F_{i,k}).$$
This says that $(\Gamma_1-\Gamma_2)^*(F_{i,k})=0$ for all $i, k\in {\Bbb N}$. We now invoke Remark \ref{basis} to conclude that $\Gamma_1=\Gamma_2$ and this completes the proof of the theorem.
\end{proof}

For operator $\Gamma$ in Proposition \ref{main} it is not
hard to check that $T_{\mathcal G}{\mathcal M}_{2,\overline{{\frak m}}}\Gamma=0$. It follows that, if $\mathcal G$ is a $HS$-Riesz basis and $\frak m$ is semi--normalized, then
$$
\mathbf{M}^{-1}_{{\frak m},\mathcal
F,\mathcal G}=\mathbf{M}_{1/{\frak m},\widetilde{\mathcal G},{{\mathcal F}^d}},
$$
for all dual $HS$-frames ${\mathcal F}^d=\{{\mathcal F}_j^d\}_{j=1}^\infty$ of $\mathcal F$.

This observation together with Proposition \ref{parsa28952} give the following result.

\begin{corollary}
Let ${\mathcal G}$ and $\mathcal F$ be $HS$--Riesz basis for $\mathcal K$ with respect
to $\mathcal H$ and let $\frak m$ be semi--normalized. Then
$\mathbf{M}_{{\frak m},\mathcal F,\mathcal G}$ is invertible on
$\mathcal K$ and $$
\mathbf{M}^{-1}_{{\frak m},\mathcal
F,\mathcal G}=\mathbf{M}_{1/{\frak m},\widetilde{\mathcal G},\widetilde{\mathcal F}}.$$
\end{corollary}

The proof of Theorem \ref{7.8.95} below relies on the following remark and proposition.

\begin{remark}\label{parsajoon}
Following \cite[Remark 2.8]{japp}, we say that the $HS$-frame ${\mathcal G}^{gd}=\{{\mathcal G}_i^{gd}\}$ is a
generalized dual $HS$-frame of ${\mathcal G}$, if $T_{\mathcal G} U_{{\mathcal G}^{gd}}$ is invertible.
It is noteworthy that with an argument
similar to the proof of \cite[Theorem 2.1]{japp} and \cite[Theorem 3.1]{von2} one can show that
the generalized
dual $HS$-frames of $\mathcal G$ are precisely the sequences
${\mathcal G}^{gd}$ such that
$${\mathcal G}^{gd}_i={\mathcal G}_iS^{-1}_{\mathcal G}Q+
\pi_i\Psi,\quad\quad\quad(i\in{\Bbb N})$$ where $\Psi$ is a bounded operator in $B({\mathcal K},\oplus{\mathcal C}_2)$ such that $T_{\mathcal G}\Psi=0$, $\pi_i:\oplus{\mathcal C}_2\rightarrow{\mathcal C}_2$ is the standard projection on the $i$-th component and $Q$ is an
invertible operator in $B({\mathcal K})$.
In what follows, the notation $GD({\mathcal G})$ is used to denote the set of all generalized $HS$-duals of $\mathcal G$.
\end{remark}


In the following result and in the sequel ${\mathrm{Inv}}({\mathcal G},{\frak m})$ refers to the set of all $HS$-Bessel sequence ${\mathcal F}$ such that the operator
$\mathbf{M}_{{\frak m},\mathcal
F,\mathcal G}$ is invertible.

\begin{proposition}\label{gelarjoon}
Suppose that $\mathcal G$ is a $HS$-frame for $\mathcal K$ with respect to $\mathcal H$ and that
$\frak m$ is a semi-normalized sequence. Then the mapping
$$\Theta:{\mathrm{Inv}}({\mathcal G},{\frak m})\rightarrow GD({\mathcal G});\quad
\{{\mathcal F}_i\}_{i=1}^\infty\mapsto\{{\mathcal G}_iS_{\mathcal G}^{-1}\mathbf{M}_{{\frak m},\mathcal
F,\mathcal G}+\pi_i{\mathcal M}_{2,\overline{{\frak m}}}\Gamma\}_{i=1}^\infty,$$
is bijective, where $\Gamma$ is the unique operator in $B({\mathcal K},\oplus_2{\mathcal C}_2)$
which satisfies the equality $\mathrm{(\ref{9895})}$.
\end{proposition}
\begin{proof}
It obviously suffices to show that $\Theta$ is an onto map. To this end, suppose that
${\mathcal G}^{gd}$ is a generalized $HS$-dual of $\mathcal G$. Then, by Remark \ref{parsajoon}, we would have a bounded operator in $B({\mathcal K},\oplus{\mathcal C}_2)$ and an invertible operator $Q$ in $B({\mathcal K})$ such that $T_{\mathcal G}\Psi=0$ and
$${\mathcal G}_i^{gd}={\mathcal G}_iS^{-1}_{\mathcal G}Q+
\pi_i\Psi,\quad\quad\quad(i\in{\Bbb N}).$$
Letting ${\mathcal F}_i:{\mathcal K}\rightarrow{\mathcal C}_2$ by
$${\mathcal F}_i=(1/\overline{{\frak m}}_i){\mathcal G}_iS_{\mathcal G}^{-1}Q^*+\pi_i({\mathcal M}_{2,1/\overline{{\frak m}}}\Psi).$$
Then ${\mathcal F}$ is a $HS$-Bessel sequence and, in particular, we have
$Q=\mathbf{M}_{{\frak m},{\mathcal
F},\mathcal G}$. Hence, ${\mathcal F}$ is in ${\mathrm{Inv}}({\mathcal G},{\frak m})$.
On the other hand, if $\Gamma$ is the unique operator which defined by ({\ref{111}}), then, for each $f\in{\mathcal K}$, we have
\begin{align*}
\Gamma(f)&=U_{{\mathcal F}}(\mathbf{M}^{-1}_{{\frak m},\mathcal
F,\mathcal G})^*(f)-{\mathcal M}_{2,1/\overline{\frak m}}U_{\mathcal G} S^{-1}_{\mathcal G}(f)\\
&={\mathcal M}_{2,1/\overline{\frak m}}U_{\mathcal G} S^{-1}_{\mathcal G}Q^*(\mathbf{M}^{-1}_{{\frak m},\mathcal
F,\mathcal G})^*(f)+{\mathcal M}_{2,1/\overline{\frak m}}\Psi(f)-{\mathcal M}_{2,1/\overline{\frak m}}U_{\mathcal G} S^{-1}_{\mathcal G}(f)\\
&={\mathcal M}_{2,1/\overline{\frak m}}\Psi(f).
\end{align*}
It follows that $\Theta({\mathcal F})={\mathcal G}^{gd}$.
\end{proof}


\begin{theorem}\label{7.8.95}
Let $\mathcal G$ be a $HS$-frame for $\mathcal K$ with respect to $\mathcal H$ and let ${\mathcal G}'$ be another $HS$-frame for $\mathcal K$ with respect to $\mathcal H$ such that $\|T_{\mathcal G}-T_{{\mathcal G}'}\|<\sqrt{A_{\mathcal G}}/2$, where $A_{\mathcal G}$ is the lower frame bound of $\mathcal G$. If $\frak m$ is semi-normalized, then there exists a one-to-one correspondence between
${\mathrm{Inv}}({\mathcal G},{\frak m})$ and ${\mathrm{Inv}}({\mathcal G}',{\frak m})$.
\end{theorem}
\begin{proof}
By Proposition \ref{gelarjoon} above, it will be enough to prove that
there exists a one-to-one correspondence between $GD({\mathcal G})$ and $GD({\mathcal G}')$. To this end, we define the map $\Lambda$ from $GD({\mathcal G})$ into $GD({\mathcal G}')$ by
$${\mathcal G}^{gd}\mapsto\{{\mathcal G}_i'S_{{\mathcal G}'}T_{\mathcal G}U_{{\mathcal G}^{gd}}+\pi_iP_{\ker(T_{{\mathcal G}'})}U_{{\mathcal G}^{gd}}\}_{i=1}^\infty,$$
where $P_{\ker(T_{{\mathcal G}'})}$
denotes the orthogonal projection of $\oplus{\mathcal C}_2$ onto ${\ker(T_{{\mathcal G}'})}$. Now, with an argument
similar to the proof of Proposition 4.15 in \cite{gita} one can show that
$\Lambda$ is bijective and this completes the proof.
\end{proof}


The next result characterizes another invertible $HS$-frame multipliers $\mathbf{M}_{{\frak m},\mathcal
F,\mathcal G} $ whose inverses can be written as $\mathbf{M}_{1/{\frak m},\widetilde{\mathcal G},\widetilde{\mathcal F}}$. In details, the following result is a generalization of a result proved by
Balazs and Stoeva \cite[Theorem 4.6]{balaz11} to $HS$-frames as well as $g$-frames. To this end, we need the following remark.

\begin{remark}\label{equivalent}
Suppose that ${\mathcal G}$ and $\mathcal F$ are $HS$-frames
for $\mathcal K$ with respect to $\mathcal H$, and
that the symbol $\frak m$ is semi--normalized.
\begin{enumerate}
\item If the $HS$-frames $\overline{{\frak m}}{\mathcal F}
:=\{\overline{{\frak m}}_i{\mathcal F}_i\}_{i=1}^\infty$ and $\mathcal G$
are equivalent; that is, there exists an invertible operator $Q$ in $B(\mathcal K)$ such that $\overline{{\frak m}}_i{\mathcal F}_i={\mathcal G}_iQ$ ($i\in{\Bbb N}$), then
$\mathbf{M}_{{\frak m},\mathcal
F,\mathcal G}$ is invertible and
$$\mathbf{M}^{-1}_{{\frak m},\mathcal
F,\mathcal G}=\mathbf{M}_{1/{\frak m},\widetilde{\mathcal G},{{\mathcal F}^d}},$$
for all dual $HS$-frames ${\mathcal F}^d$ of $\mathcal F$. Indeed, on the one hand we have
$$
\mathbf{M}_{{\frak m},\mathcal F,\mathcal G}=T_{\mathcal F}{\mathcal M}_{2,{\frak m}}U_{\mathcal G}
=T_{\overline{{\frak m}}{\mathcal F}}U_{\mathcal G}=Q^*S_{\mathcal G},
$$
and on the other hand, if the letter ${\mathcal F}Q^{-1}$ respectively $(1/\overline{{\frak m}}){\mathcal G}$ refer to the $HS$-Bessel sequence $\{{\mathcal F}_iQ^{-1}\}_{i=1}^\infty$ respectively $\{(1/\overline{{\frak m}}_i){\mathcal G}_i\}_{i=1}^\infty$, then we observe that
\begin{align*}
\mathbf{M}_{1/{\frak m},\widetilde{\mathcal G},{\mathcal F}^d}&=T_{\widetilde{\mathcal G}}{\mathcal M}_{2,1/{\frak m}}U_{{\mathcal F}^d}\\
&=S^{-1}_{\mathcal G}T_{(1/\overline{{\frak m}}){\mathcal G}}U_{{\mathcal F}^d}\\
&=S^{-1}_{\mathcal G}T_{{\mathcal F}Q^{-1}}U_{{\mathcal F}^d}\\
&=(Q^*S_{\mathcal G})^{-1}.
\end{align*}
Conversely, if
$\mathbf{M}_{{\frak m},\mathcal
F,\mathcal G}$ is invertible and
$\mathbf{M}^{-1}_{{\frak m},\mathcal
F,\mathcal G}=\mathbf{M}_{1/{\frak m},\widetilde{\mathcal G},{{\mathcal F}^d}},$
for all dual $HS$-frames ${\mathcal F}^d$ of $\mathcal F$, then Theorem \ref{main}
implies that $\Gamma^*U_{{\mathcal F}^d}=0$. Now, with an argument
similar to the proof of Theorem \ref{main} one can conclude that
$\Gamma=0$. From this, by Eq. (\ref{111}), we deduce
that
$$\overline{{\frak m}}_i{\mathcal F}_i={\mathcal G}_iS^{-1}_{\mathcal G}\mathbf{M}_{\overline{{\frak m}},\mathcal
G,\mathcal F}\quad\quad\quad(i\in{\Bbb N}),$$
and thus the $HS$-frames $\mathcal G$ and $\overline{{\frak m}}{\mathcal F}$ are
equivalent. Hence, we can give the interpretation below for equivalent $HS$-frames:\\
``{\it If $\frak m$ is semi--normalized, then $\mathbf{M}_{{\frak m},\mathcal
F,\mathcal G}$ is invertible and
$\mathbf{M}^{-1}_{{\frak m},\mathcal
F,\mathcal G}=\mathbf{M}_{1/{\frak m},\widetilde{\mathcal G},{{\mathcal F}^d}},$
for all dual $HS$-frames ${\mathcal F}^d$ of $\mathcal F$ if and only if
the $HS$-frames $\overline{{\frak m}}{\mathcal F}:=\{\overline{{\frak m}}_i{\mathcal F}_i\}_{i=1}^\infty$ and $\mathcal G$ are equivalent.}"
\item  If $\mathcal Y$
denotes any one of the $HS$-frames $\mathcal G$ and $\mathcal F$, then it is easy to check that
$$\oplus{\mathcal C}_2={{\mathrm{ran}}}(U_{\mathcal Y})\oplus\ker( T_{\mathcal Y})$$ and thus $P_{\ker
(T_{\mathcal Y})}+P_{{{\mathrm{ran}}}(U_{\mathcal Y})}=Id_{\oplus{\mathcal C}_2}$, where $P_X$
denotes the orthogonal projection of $\oplus{\mathcal C}_2$ onto $X$. In
particular, we have
$$P_{{{\mathrm{ran}}}(U_{\mathcal G})}=U_{\mathcal G} T_{\widetilde{{\mathcal G}}}\quad\quad\quad{\hbox{and}}\quad\quad\quad P_{{{\mathrm{ran}}}(U_{\widetilde{{\mathcal F}}})}=U_{\widetilde{{\mathcal F}}} T_{\mathcal F}.$$
\end{enumerate}
\end{remark}

The following proposition is now immediate.


\begin{proposition}\label{main}
Suppose that ${\mathcal G}$ and $\mathcal F$ are $HS$-frames
for $\mathcal K$ with respect to $\mathcal H$, and
that there exists a non-zero constant $c$ such that ${\frak m}_i={c}$ for all $i\in {\Bbb N}$.
Then the following statements are equivalent.
\begin{enumerate}
\item\label{condition1} $\mathbf{M}_{{\frak m},\mathcal
F,\mathcal G}$ is invertible and
$\mathbf{M}^{-1}_{{\frak m},\mathcal
F,\mathcal G}=\mathbf{M}_{1/{\frak m},\widetilde{\mathcal G},\widetilde{\mathcal F}}.$
\item $\mathcal F$ and $\mathcal G$ are equivalent $HS$-frames.
\end{enumerate}
\end{proposition}
\begin{proof}
We first note that, without loss of generality, we may consider $c=1$.
The necessity of the condition ``$\mathcal F$ and $\mathcal G$ are equivalent $HS$-frames" follows from part (1) of Remark \ref{equivalent}. We prove its sufficiency.
To this end, suppose that the condition (\ref{condition1}) is satisfied. From this, we observe that
$$U_{\mathcal G}T_{\widetilde{\mathcal G}}=U_{\mathcal G}\mathbf{M}_{1/{\frak m},\widetilde{\mathcal G},\widetilde{\mathcal F}}\mathbf{M}_{{\frak m},\mathcal
F,\mathcal G}T_{\widetilde{\mathcal G}}=U_{\mathcal G}T_{\widetilde{\mathcal G}}U_{\widetilde{\mathcal F}}T_{\mathcal F}U_{\mathcal G}T_{\widetilde{\mathcal G}}.$$
We now invoke part (2) of Remark \ref{equivalent} to conclude that $$P_{{\mathrm{ran}}(U_{\mathcal G})}P_{{\mathrm{ran}}(U_{\widetilde{\mathcal F}})}P_{{\mathrm{ran}}(U_{\mathcal G})}=P_{{\mathrm{ran}}(U_{\mathcal G})}.$$
This says that ${{\mathrm{ran}}(U_{\widetilde{\mathcal F}})}\subseteq {{\mathrm{ran}}(U_{\mathcal G})}$ and thus
$${{\mathrm{ran}}(U_{{\mathcal F}})}\subseteq {{\mathrm{ran}}(U_{\mathcal G})}.$$
Analogously, one can show that the reverse inclusion is also true. It follows that ${\mathrm{ran}}(U_{\mathcal F})={\mathrm{ran}}(U_{\mathcal G})$. Now, we follow the proof of \cite[Lemma 2.1]{balan} to show that there exists an invertible operator $Q\in B({\mathcal K})$ such that
${\mathcal G}_i={\mathcal F}_iQ$ ($i\in {\Bbb N}$).
To this end, suppose that
${\mathcal G}'=\{{\mathcal G}_iS_{{\mathcal G}}^{-{\frac{1}{2}}}\}$ and ${\mathcal F}'=\{{\mathcal F}_iS_{{\mathcal F}}^{-{\frac{1}{2}}}\}$.
Observe that ${\mathcal G}'$ and ${\mathcal F}'$ are $HS$-frames with lower and upper frame bounds 1 and thus $S_{{\mathcal G}'}=Id_{{\mathcal K}}=S_{{\mathcal F}'}$. Moreover, it is not hard to check that ${\mathrm{ran}}(U_{{\mathcal F}'})={\mathrm{ran}}(U_{{\mathcal G}'})$
and thus
$$
U_{{\mathcal G}'}T_{{\mathcal G}'}=P_{{\mathrm{ran}}(U_{{\mathcal G}'})}=
P_{{\mathrm{ran}}(U_{{\mathcal F}'})}=U_{{\mathcal F}'}T_{{\mathcal F}'}.$$
Hence, if we set $Q=T_{{\mathcal G}'}U_{{\mathcal F}'}$, then
$$Q^*Q=T_{{\mathcal F}'}U_{{\mathcal G}'}T_{{\mathcal G}'}U_{{\mathcal F}'}=Id_{\mathcal K}.$$
This says that $Q$ is an isometry, $Q^*(f)=\sum_{i=1}^\infty{{\mathcal F}'_i}^*{\mathcal G}'_i(f)$ and
$$f=\sum_{i=1}^\infty{{\mathcal F}'_i}^*{\mathcal G}'_iQ(f)\quad\quad\quad\quad\quad(f\in{\mathcal K}).$$
Now, for each $f\in{\mathcal K}$, we have
\begin{align*}
\sum_{i=1}^\infty\|{\mathcal G}'_iQ(f)-{\mathcal F}'_i(f)\|^2_{{\mathcal C}_2}&=
\sum_{i=1}^\infty\|{\mathcal G}'_iQ(f)\|^2_{{\mathcal C}_2}+\sum_{i=1}^\infty\|{\mathcal F}'_i(f)\|^2_{{\mathcal C}_2}\\&-2Re{\Big(}{\Big[}\sum_{i=1}^\infty{{\mathcal F}'_i}^*{\mathcal G}'_iQ,f{\Big]}_{\mathbf{tr}}{\Big)}\\
&=\|Q(f)\|^2-\|f\|^2=0
\end{align*}
It follows that ${\mathcal G}'_iQ={\mathcal F}'_i$ ($i\in{\Bbb N}$).
Having reached this state it remains to prove that $Q$ is onto or equivalently $\ker(Q^*)=\{0\}$. To this end, suppose that $g\in \ker(Q^*)$. Hence, $U_{{\mathcal G}'}(g)\in\ker(T_{{\mathcal F}'})$ and thus, since
$$\ker(T_{{\mathcal F}'})={\mathrm{ran}}(U_{{\mathcal F}'})^\perp={\mathrm{ran}}(U_{{\mathcal G}'})^\perp=\ker(T_{{\mathcal G}'}),$$
we can deduce that $U_{{\mathcal G}'}(g)\in\ker(T_{{\mathcal G}'})$. It follows that
$g=T_{{\mathcal G}'}U_{{\mathcal G}'}(g)=0$. We have now completed the proof of the
proposition.
\end{proof}


We conclude this work by the following result which is of interest in its own right.
In details, as an another application of Theorem \ref{main}, we have the following
surprising new results about dual $HS$-frames as well as dual $g$-frames which shows that a $HS$-frame [respectively, $g$-frame] is uniquely determined by the set of its dual $HS$-frames [respectively, $g$-frames].

\begin{theorem}
Let $\mathcal G$ and $\mathcal F$ be $HS$-frames for $\mathcal K$ with respect to $\mathcal H$. If every dual $HS$-frame ${\mathcal G}^d$ of $\mathcal G$ is a dual $HS$-frame of $\mathcal F$, then ${\mathcal G}={\mathcal F}$.
\end{theorem}
\begin{proof}
The assumption implies that $(T_{\mathcal G}-T_{\mathcal F})U_{{\mathcal G}^d}=0$ for all dual $HS$-frames ${\mathcal G}^d$ of $\mathcal G$. Hence, with an argument
similar to the proof of Theorem \ref{main} one can show that $T_{\mathcal G}=T_{\mathcal F}$. Whence ${\mathcal G}={\mathcal F}$.
\end{proof}

\end{document}